\documentclass[reqno,12pt]{article} %leqno is the option to put formula numbers on the left side
\usepackage{amsmath, amssymb}
\usepackage{amsthm} %theorem environment option
\usepackage{pictexwd,dcpic}
\usepackage{multirow}
\newcommand{\Mod}[1]{\ (\textup{mod}\ #1)}
\theoremstyle{plain} %text of this environment is typesetted in italics
\newtheorem{theorem}{\indent\sc Theorem}[section]
\newtheorem{lemma}[theorem]{\indent\sc Lemma}
\newtheorem{corollary}[theorem]{\indent\sc Corollary}
\newtheorem{proposition}[theorem]{\indent\sc Proposition}

\theoremstyle{definition} %text of this environment is typesetted in roman letters

\newtheorem{remark}[theorem]{\indent\sc Remark}

\makeatletter
\def\address#1#2{\begingroup
\noindent\parbox[t]{7.8cm}{%
\small{\scshape\ignorespaces#1}\par\vskip1ex
\noindent\small{\itshape E-mail address}%
\/: #2\par\vskip4ex}\hfill%
\endgroup}%
\makeatother
\title{Generation of class fields by using the Weber function} %title of the paper
\author{
\textsc{Ja Kyung Koo, Dong Hwa Shin and Dong Sung Yoon} %names of authors
}
\date{} %leave empty
\begin{document}

\allowdisplaybreaks

\maketitle

%%%%%%%%%%%%%%% footnote %%%%%%%%%%%%%%%%
\footnote{ %2010 MSC numbers
}
\footnote{ %key words and phrases
} \footnote{
\thanks{
The second named author was supported by Hankuk University of Foreign Studies Research Fund of
2014. The third named author was supported by the National Institute for Mathematical Sciences, Republic of Korea.} }
%%%%%%%%%%%%%%%%%%%%%%%%%%%%%%%%%%%%%%%%%

\begin{abstract}
Let $K$ be an imaginary quadratic field
and $\mathcal{O}_K$ be its ring of integers.
Let $h_E$ be the Weber function on certain elliptic curve $E$ with complex multiplication by $\mathcal{O}_K$.
We show that if $N$ ($>1$) is an integer prime to $6$, then the function $h_E$ alone
generates the ray class field modulo $N\mathcal{O}_K$ over $K$
when evaluated at some $N$-torsion point of $E$,
which would be a partial answer to the question mentioned in \cite[p.134]{Silverman}.
\end{abstract}

\maketitle

\section {Introduction}

For a lattice $\Lambda$ in $\mathbb{C}$, let
\begin{equation*}
g_2(\Lambda)=60\sum_{\lambda\in\Lambda\setminus\{0\}}\frac{1}{\lambda^4},\quad
g_3(\Lambda)=140\sum_{\lambda\in\Lambda\setminus\{0\}}\frac{1}{\lambda^6},\quad
\Delta(\Lambda)=g_2(\Lambda)^3-27g_3(\Lambda)^2
\end{equation*}
and
\begin{equation*}
j(\Lambda)=1728\frac{g_2(\Lambda)^3}{\Delta(\Lambda)}.
\end{equation*}
Note that the series $g_2(\Lambda)$ and $g_3(\Lambda)$ converge absolutely, and $\Delta(\Lambda)\neq0$
(\cite[Lemma 10.2 and Proposition 10.7]{Cox}).
We further define the \textit{Weierstrass $\wp$-function} relative to $\Lambda$ by
\begin{equation*}
\wp(z;\Lambda)=\frac{1}{z^2}+\sum_{\lambda\in\Lambda\setminus\{0\}}
\left(\frac{1}{(z-\lambda)^2}-\frac{1}{\lambda^2}\right)\quad(z\in\mathbb{C})
\end{equation*}
with derivative $\wp'(z;\Lambda)=d\wp(z;\Lambda)/dz$.
This is an even function which is periodic with respect to $\Lambda$.
\par
Let $K$ be an imaginary quadratic field and $\mathcal{O}_K$ be its
ring of integers.
Let $E$ be the elliptic curve over $\mathbb{C}$ with
parametrization
\begin{equation}\label{parametrization}
\begin{array}{ccl}
\varphi_E~:~\mathbb{C}/\mathcal{O}_K&\stackrel{\sim}{\rightarrow}
&E(\mathbb{C})~:~y^2=4x^3-g_2(\mathcal{O}_K)x-g_3(\mathcal{O}_K)\\
\qquad z&\mapsto&(\wp(z;\mathcal{O}_K),\wp'(z;\mathcal{O}_K))
\end{array}
\end{equation}
so that it has complex multiplication by $\mathcal{O}_K$ with $j$-invariant $j(E)=j(\mathcal{O}_K)$
(see \cite[$\S$14]{Cox} or \cite[Chapter 6]{Silverman2}).
Then the map $h_E:E\rightarrow\mathbb{P}^1$ defined by
\begin{equation}\label{Weber}
h_E(\varphi_E(z))=\left\{\begin{array}{ll}
\displaystyle\frac{g_2(\mathcal{O}_K)g_3(\mathcal{O}_K)}{\Delta(\mathcal{O}_K)}
\wp(z;\mathcal{O}_K) & \textrm{if}~j(E)\neq0,1728,\phantom{\Bigg|}\\
\displaystyle\frac{g_2(\mathcal{O}_K)^2}{\Delta(\mathcal{O}_K)}
\wp(z;\mathcal{O}_K)^2 & \textrm{if}~j(E)=1728,\\
\displaystyle\frac{g_3(\mathcal{O}_K)}{\Delta(\mathcal{O}_K)}
\wp(z;\mathcal{O}_K)^3 & \textrm{if}~j(E)=0\phantom{\Bigg|}
\end{array}\right.
\end{equation}
gives rise to an isomorphism from $E/\mathrm{Aut}(E)$ onto $\mathbb{P}^1$
(\cite[Chapter II, Example 5.5.2]{Silverman}), which is called the
\textit{Weber function} on $E$.
\par
Let $\mathfrak{f}$ be a proper nontrivial ideal of $\mathcal{O}_K$ and
$K_\mathfrak{f}$ be the ray class field of $K$ modulo $\mathfrak{f}$ (see $\S$\ref{Sectclass}).
Hasse (\cite{Hasse}) showed by utilizing the theory of complex multiplication that
\begin{equation}\label{Hassegenerators}
K_\mathfrak{f}=K(j(E),h_E(E[\,\mathfrak{f}\,]))
\end{equation}
where $E[\,\mathfrak{f}\,]$ is the group of $\mathfrak{f}$-torsion points of $E$.
However,
it is still an open question
whether $h_E(P)$ alone
generates $K_\mathfrak{f}$ over $K$
for some $P\in E[\,\mathfrak{f}\,]$ as far as we understand.
\par
In this paper we shall give a partial answer to this question.
More precisely,
let $\mathfrak{f}=N\mathcal{O}_K$ for an integer $N>1$ prime to $6$,
and let $K_\mathcal{O}$ be
the ring class field of the order $\mathcal{O}$ of conductor $N$ in $K$ (see $\S$\ref{Sectclass}).
We shall show that
if $K$ is different from $\mathbb{Q}(\sqrt{-1})$ and $\mathbb{Q}(\sqrt{-3})$, then
the relative norm
\begin{equation*}
\mathrm{N}_{K_\mathfrak{f}/K_\mathcal{O}}(h_E(\varphi_E(2/N))-h_E(\varphi_E(1/N)))
=\prod_{t\in(\mathbb{Z}/N\mathbb{Z})^\times/\{\pm1\}}
(h_E(\varphi_E(2t/N))-h_E(\varphi_E(t/N)))
\end{equation*}
generates $K_\mathcal{O}$ over $K$ (Theorem \ref{relativenorm} and Remark \ref{explicit}).
It then follows that $h_E(\varphi_E(1/N))$ indeed generates
$K_\mathfrak{f}$ over $K$ (Corollary \ref{main}). For this
we shall introduce Fricke invariants and
use the second Kronecker's limit formula concerning Siegel-Ramachandra invariants.

\section {Fricke and Siegel functions}

Let $\mathbb{H}=\{\tau\in\mathbb{C}~|~\mathrm{Im}(\tau)>0\}$ be the complex
upper half-plane.
We define the \textit{elliptic modular function} $j(\tau)$ on $\mathbb{H}$ by
\begin{equation}\label{j}
j(\tau)=j([\tau,1])=1728\frac{g_2([\tau,1])^3}{\Delta([\tau,1])}.
\end{equation}
Furthermore, for $k\in\{1,2,3\}$ and $\left[\begin{smallmatrix}r_1\\r_2\end{smallmatrix}\right]\in\mathbb{Q}^2\setminus\mathbb
{Z}^2$ we define the $k$-th
\textit{Fricke function} $f^{(k)}_{\left[\begin{smallmatrix}r_1\\r_2\end{smallmatrix}\right]}(\tau)$ on $\mathbb{H}$ by
\begin{equation}\label{Fricke}
f^{(k)}_{\left[\begin{smallmatrix}r_1\\r_2\end{smallmatrix}\right]}(\tau)=
\left\{\begin{array}{ll}
\displaystyle\frac{g_2([\tau,1])g_3([\tau,1])}{\Delta([\tau,1])}\wp(r_1\tau+r_2;[\tau,1])
& \textrm{if}~k=1,\phantom{\Bigg|}\\
\displaystyle\frac{g_2([\tau,1])^2}{\Delta([\tau,1])}\wp(r_1\tau+r_2;[\tau,1])^2
& \textrm{if}~k=2,\phantom{\Bigg|}\\
\displaystyle\frac{g_3([\tau,1])}{\Delta([\tau,1])}\wp(r_1\tau+r_2;[\tau,1])^3
& \textrm{if}~k=3\phantom{\Bigg|}
\end{array}\right.
\end{equation}
which is holomorphic on $\mathbb{H}$ and depends only on $\pm\left[\begin{smallmatrix}r_1\\r_2\end{smallmatrix}\right]\Mod{\mathbb{Z}^2}$
(\cite[$\S$6.1.A]{Shimura}).
Here, for the sake of convenience, we often write $f_{\left[\begin{smallmatrix}r_1\\r_2\end{smallmatrix}\right]}(\tau)$
for $f^{(1)}_{\left[\begin{smallmatrix}r_1\\r_2\end{smallmatrix}\right]}(\tau)$.
\par
Let
\begin{equation*}
\mathcal{F}_N=\left\{\begin{array}{ll}\mathbb{Q}(j(\tau)) & \textrm{if}~N=1,\\
\mathbb{Q}(j(\tau),f_{\left[\begin{smallmatrix}r_1\\r_2\end{smallmatrix}\right]}
(\tau)~|~\left[\begin{smallmatrix}r_1\\r_2\end{smallmatrix}\right]\in(1/N)\mathbb{Z}^2
\setminus\mathbb{Z}^2) & \textrm{if}~N\geq2.
\end{array}\right.
\end{equation*}
Then it is well-known that
$\mathcal{F}_N$ is a Galois extension of $\mathcal{F}_1$ whose
Galois group is isomorphic to $\mathrm{GL}_2(\mathbb{Z}/N\mathbb{Z})/\{\pm I_2\}$
(\cite[Theorem 6.6 (1) and (2)]{Shimura}).
In fact,
$\mathcal{F}_N$ coincides with the field of all meromorphic modular functions
of level $N$ whose Fourier expansions with respect to $q^{1/N}$ ($q=e^{2\pi i\tau}$) have coefficients in
the $N$-th cyclotomic field  (\cite[Proposition 6.9 (1)]{Shimura}).
\par
Let  $g_{\left[\begin{smallmatrix}r_1\\r_2\end{smallmatrix}\right]}(\tau)$
be the
\textit{Siegel function} on $\mathbb{H}$ given by the following infinite product
\begin{equation*}
g_{\left[\begin{smallmatrix}r_1\\r_2\end{smallmatrix}\right]}(\tau)
=-e^{\pi i r_2(r_1-1)}q^{(1/2)\mathbf{B}_2(r_1)}
(1-q^{r_1}e^{2\pi ir_2})\prod_{n=1}^\infty (1-q^{n+r_1}e^{2\pi ir_2})
(1-q^{n-r_1}e^{-2\pi ir_2}),
\end{equation*}
where $\mathbf{B}_2(X)=X^2-X+1/6$
is the second Bernoulli polynomial.
Observe that it has neither zeros nor poles on $\mathbb{H}$, and
its $12N$-th power belongs to $\mathcal{F}_N$ (\cite[Chapter 2, Theorem 1.2]{K-L}).

\begin{lemma}\label{FrickeSiegel}
Let $\mathbf{u},\mathbf{v}\in\mathbb{Q}^2\setminus\mathbb{Z}^2$ such that
$\mathbf{u}\not\equiv\pm\mathbf{v}\Mod{\mathbb{Z}^2}$. Then
we have the relation
\begin{equation*}
(f_\mathbf{u}(\tau)-
f_\mathbf{v}(\tau))^6=\frac{j(\tau)^2(j(\tau)-1728)^3}{2^{30}3^{24}}
\frac{g_{\mathbf{u}+\mathbf{v}}(\tau)^6g_{\mathbf{u}-\mathbf{v}}(\tau)^6}
{g_\mathbf{u}(\tau)^{12}g_\mathbf{v}(\tau)^{12}}.
\end{equation*}
\end{lemma}
\begin{proof}
 It follows from \cite[Chapter 18, Theorem 2]{Lang}, \cite[p.29 and p.51]{K-L} and definitions (\ref{j}) and (\ref{Fricke}).
\end{proof}

\section {Class fields over imaginary quadratic fields}\label{Sectclass}

Let $K$ be an imaginary quadratic field of discriminant $d_K$ and $\mathcal{O}_K$ be its
ring of integers. If we set
\begin{equation*}
\tau_K=\frac{d_K+\sqrt{d_K}}{2},
\end{equation*}
then we see that $\tau_K\in\mathbb{H}$ and $\mathcal{O}_K=[\tau_K,1]$
(\cite[p.103]{Cox}).
\par
Let $\mathfrak{f}$ be a nontrivial ideal of $\mathcal{O}_K$. We denote by
\begin{eqnarray*}
I_K(\mathfrak{f})&=&\textrm{the group of all fractional ideals of $K$ prime to}~\mathfrak{f},\\
P_K(\mathfrak{f})&=&\langle\alpha\mathcal{O}_K~|~\alpha\in\mathcal{O}_K\setminus\{0\}~
\textrm{such that $\alpha\mathcal{O}_K$ is prime to}~\mathfrak{f}\rangle,\\
P_{K,1}(\mathfrak{f})&=&\langle\alpha\mathcal{O}_K~|~\alpha\in\mathcal{O}_K\setminus\{0\}~
\textrm{such that}~\alpha\mathcal{O}_K~\textrm{is prime to $\mathfrak{f}$ and}~\alpha\equiv1\Mod{\mathfrak{f}}\rangle,
\end{eqnarray*}
and call the quotient group
\begin{equation*}
\mathrm{Cl}(\mathfrak{f})=I_K(\mathfrak{f})/P_{K,1}(\mathfrak{f})
\end{equation*}
the \textit{ray class group} of $K$ modulo $\mathfrak{f}$.
By the existence theorem of class field theory there is a unique finite abelian extension
$K_\mathfrak{f}$ of $K$, called the \textit{ray class field} of $K$ modulo $\mathfrak{f}$, such that
\begin{equation*}
\mathrm{Gal}(K_\mathfrak{f}/K)\simeq\mathrm{Cl}(\mathfrak{f})
\end{equation*}
via the Artin reciprocity map for $\mathfrak{f}$ (\cite[Chapter V, Theorem 9.1]{Janusz}).
In particular, we call $K_{\mathcal{O}_K}$ the \textit{Hilbert class field} of $K$
and denote it also by $H_K$.

\begin{proposition}\label{CM}
Let $E$ be the elliptic curve given in \textup{(\ref{parametrization})} and $h_E$ be
the Weber function on $E$ defined as in \textup{(\ref{Weber})}. Then,
\begin{itemize}
\item[\textup{(i)}] $H_K$ is generated by $j(E)=j(\tau_K)$ over $K$.
\item[\textup{(ii)}] If $z_0\in\mathbb{C}$ is a generator
of $\mathfrak{f}^{-1}/\mathcal{O}_K$ as an $\mathcal{O}_K$-module, then
$h_E(\varphi_E(z_0))$ generates $K_\mathfrak{f}$ over $H_K$.
\end{itemize}
\end{proposition}
\begin{proof}
(i) See \cite[Chapter 10, Theorem 1]{Lang}.\\
(ii) See \cite[Chapter 10, Corollary to Theorem 7]{Lang}.
\end{proof}

\begin{remark}\label{WeberFricke}
\begin{itemize}
\item[(i)] Note that $j(\tau)$ gives rise to a bijection
$\mathrm{SL}_2(\mathbb{Z})\backslash\mathbb{H}
\rightarrow\mathbb{C}$, and
\begin{equation*}
j(\tau_K)=\left\{
\begin{array}{ll}
1728 & \textrm{if}~K=\mathbb{Q}(\sqrt{-1}),\\
0 & \textrm{if}~K=\mathbb{Q}(\sqrt{-3})
\end{array}\right.
\end{equation*}
(\cite[Theorem 11.2 and p.261]{Cox}).
\item[(ii)] If $\left[\begin{smallmatrix}r_1\\r_2\end{smallmatrix}\right]\in\mathbb{Q}^2\setminus\mathbb{Z}^2$, then one can readily attain
\begin{equation*}
h_E(\varphi_E(r_1\tau_K+r_2))=f^{(k)}_{\left[\begin{smallmatrix}r_1\\r_2\end{smallmatrix}\right]}(\tau_K)
\end{equation*}
where
\begin{equation*}
k=\left\{
\begin{array}{ll}
1 & \textrm{if}~K\neq\mathbb{Q}(\sqrt{-1}),\mathbb{Q}(\sqrt{-3}),\\
2 & \textrm{if}~K=\mathbb{Q}(\sqrt{-1}),\\
3 & \textrm{if}~K=\mathbb{Q}(\sqrt{-3}).
\end{array}\right.
\end{equation*}
\end{itemize}
\end{remark}

Now, let $\mathfrak{f}=N\mathcal{O}_K$
for a positive integer $N$ and $P_{K,\mathbb{Z}}(\mathfrak{f})$ be the subgroup of $I_K(\mathfrak{f})$
given by
\begin{align*}
P_{K,\mathbb{Z}}(\mathfrak{f})=\langle\alpha\mathcal{O}_K~|~
\alpha\in\mathcal{O}_K\setminus\{0\}~\textrm{such that $\alpha\mathcal{O}_K$ is prime to $\mathfrak{f}$ and}~
\alpha\equiv a\Mod{\mathfrak{f}}\\\textrm{for some integer $a$ with}~
\gcd(N,a)=1\rangle.
\end{align*}
Let $\mathcal{O}=[N\tau_K,1]$ be the order of conductor $N$ in $K$,
$I(\mathcal{O})$ be the group of all proper fractional $\mathcal{O}$-ideals and
$P(\mathcal{O})$ be the subgroup of $I(\mathcal{O})$ consisting of principal $\mathcal{O}$-ideals.
Since
$I(\mathcal{O})/P(\mathcal{O})\simeq I_K(\mathfrak{f})/P_{K,\mathbb{Z}}(\mathfrak{f})$
(\cite[Proposition 7.22]{Cox}), there is a unique abelian extension $K_\mathcal{O}$ of $K$,
called the \textit{ring class field} of the order $\mathcal{O}$,
 such that
\begin{equation*}
\mathrm{Gal}(K_\mathcal{O}/K)\simeq I_K(\mathfrak{f})/P_{K,\mathbb{Z}}(\mathfrak{f})
\simeq I(\mathcal{O})/P(\mathcal{O}).
\end{equation*}
As a consequence of the main theorem of complex multiplication we have
\begin{equation*}
K_\mathcal{O}=K(j(\mathcal{O}))=K(j(N\tau_K))
\end{equation*}
(\cite[Chapter 10, Theorem 5]{Lang}).
\par
Since
\begin{equation*}
\mathrm{Gal}(H_K/K)\simeq\mathrm{Cl}(\mathcal{O}_K)=
I_K(\mathcal{O}_K)/P_K(\mathcal{O}_K)\simeq
I_K(\mathfrak{f})/P_K(\mathfrak{f}),
\end{equation*}
we obtain
\begin{equation}\label{Galois}
\mathrm{Gal}(K_\mathfrak{f}/K_\mathcal{O})\simeq P_{K,\mathbb{Z}}(\mathfrak{f})/P_{K,1}(\mathfrak{f}).
\end{equation}
Furthermore, the inclusions $P_{K,1}(\mathfrak{f})\subseteq P_{K,\mathbb{Z}}(\mathfrak{f})
\subseteq P_K(\mathfrak{f})$ imply
\begin{equation}\label{classinclusion}
H_K\subseteq K_\mathcal{O}\subseteq K_\mathfrak{f}.
\end{equation}

\section {Fricke invariants}

For an integer $N>1$, let
\begin{equation*}
V_N=\{\mathbf{v}\in\mathbb{Q}^2~|~\mathbf{v}~\textrm{has primitive denominator}~N\}.
\end{equation*}
We call a family $\{h_\mathbf{v}(\tau)\}_{\mathbf{v}\in V_N}$ of
functions in $\mathcal{F}_N$ a \textit{Fricke family} of level $N$ if for
every $\mathbf{v}\in V_N$
\begin{itemize}
\item $h_\mathbf{v}(\tau)$ is holomorphic on $\mathbb{H}$,
\item $h_\mathbf{v}(\tau)$ depends only on $\pm\mathbf{v}\Mod{\mathbb{Z}^2}$,
\item $h_\mathbf{v}(\tau)^\gamma
=h_{\gamma^T\mathbf{v}}(\tau)$ for all $\gamma\in\mathrm{Gal}(\mathbb{Z}/N\mathbb{Z})/\{\pm I_2\}\simeq\mathrm{Gal}(\mathcal{F}_N/\mathcal{F}_1)$,
where $\gamma^T$ stands for the transpose of $\gamma$.
\end{itemize}
In particular, $\{f_\mathbf{v}(\tau)\}_{\mathbf{v}\in V_N}$ and $\{g_{\mathbf{v}}(\tau)^{12N}\}_{\mathbf{v}\in V_N}$
are both Fricke families of level $N$
(\cite[Theorem 6.6 (2)]{Shimura} and \cite[Chapter 2, Proposition 1.3]{K-L}).
\par
Let $K$ be an imaginary quadratic field and
$\mathfrak{f}$ be a proper nontrivial ideal of $\mathcal{O}_K$.
Let $N$ ($>1$) be the
smallest positive integer in $\mathfrak{f}$, and let $C\in\mathrm{Cl}(\mathfrak{f})$.
Take an ideal $\mathfrak{c}$ of $\mathcal{O}_K$ in $C$ and set
\begin{eqnarray*}
\mathfrak{f}\mathfrak{c}^{-1}&=&[\omega_1,\omega_2]\quad\textrm{for some}~ \omega_1,\omega_2\in\mathbb{C}~\textrm{such that}~\omega=\omega_1/\omega_2\in\mathbb{H},\\
1&=&r_1\omega_1+r_2\omega_2\quad\textrm{for some}~r_1,r_2\in(1/N)\mathbb{Z}.
\end{eqnarray*}
With the above notations let $\{h_\mathbf{v}(\tau)\}_{\mathbf{v}\in V_N}$ be a Fricke family of level $N$.
We then define the \textit{Fricke invariant} $h_\mathfrak{f}(C)$ modulo $\mathfrak{f}$ at the class $C$ by
\begin{equation}\label{Frickeinvariant}
h_\mathfrak{f}(C)=h_{\left[\begin{smallmatrix}r_1\\r_2\end{smallmatrix}\right]}(\omega).
\end{equation}
This value depends only on $\mathfrak{f}$ and $C$, not on the choices of $\mathfrak{c}$, $\omega_1$ and $\omega_2$ (\cite[Chapter 11, $\S$1]{K-L}).

\begin{proposition}\label{invariant}
The Fricke invariant $h_\mathfrak{f}(C)$ lies in $K_\mathfrak{f}$.
And, we have the transformation formula
\begin{equation*}
h_\mathfrak{f}(C)^{\sigma_\mathfrak{f}(C')}=h_\mathfrak{f}(CC')\quad\textrm{for any}~C'\in\mathrm{Cl}(\mathfrak{f}),
\end{equation*}
where $\sigma_\mathfrak{f}:\mathrm{Cl}(\mathfrak{f})\stackrel{\sim}{\rightarrow}\mathrm{Gal}(K_\mathfrak{f})$
is the Artin reciprocity map for $\mathfrak{f}$.
\end{proposition}
\begin{proof}
See \cite[Chapter 11, Theorem 1.1]{K-L}.
\end{proof}

\begin{remark}
\begin{itemize}
\item[(i)] Jung et al. (\cite{J-K-S}) recently showed that if $K$ is different from $\mathbb{Q}(\sqrt{-1})$ and $\mathbb{Q}(\sqrt{-3})$, then
    $f_\mathfrak{f}(C)$ generates $K_\mathfrak{f}$ over $H_K$.
\item[(ii)] For simplicity, we mean
by $g_\mathfrak{f}(C)$
the Fricke invariant obtained from
the family $\{g_\mathbf{v}(\tau)^{12N}\}_{\mathbf{v}\in V_N}$, and call it
the \textit{Siegel-Ramachandra invariant} modulo $\mathfrak{f}$ at $C$.
\item[(iii)]
Let $\mathfrak{f}=\prod_{\mathfrak{p}\,|\,\mathfrak{f}}\mathfrak{p}^{e_\mathfrak{p}}$
be the prime ideal factorization of $\mathfrak{f}$ and
$G_\mathfrak{p}=(\mathcal{O}_K/\mathfrak{p}^{e_\mathfrak{p}})^\times
/\{\alpha+\mathfrak{p}^{e_\mathfrak{p}}\in
(\mathcal{O}_K/\mathfrak{p}^{e_\mathfrak{p}})^\times~|~\alpha\in\mathcal{O}_K^\times\}$
for each $\mathfrak{p}\,|\,\mathfrak{f}$.
Schertz (\cite{Schertz}) first proved that
if
\begin{equation}\label{condition}
\textrm{the exponent of $G_\mathfrak{p}$ is greater than $2$ for every $\mathfrak{p}\,|\,\mathfrak{f}$},
\end{equation}
then
$g_\mathfrak{f}(C)$ generates $K_\mathfrak{f}$ over $K$.
And, Koo and Yoon further improved in \cite{K-Y} that
one can replace the exponent of $G_\mathfrak{p}$ in the assumption (\ref{condition})
by the order of $G_\mathfrak{p}$.
\end{itemize}
\end{remark}

\section {The second Kronecker's limit formula}

Let $K$ be an imaginary quadratic field and
$\mathfrak{f}$ be a proper nontrivial ideal of $\mathcal{O}_K$.
Let $\chi$ be a nonprincipal character of $\mathrm{Cl}(\mathfrak{f})$.
We define the \textit{Stickelberger
element} $S_\mathfrak{f}(\chi)$ by
\begin{equation}\label{Stickelberger}
S_\mathfrak{f}(\chi)=
\sum_{C\in\mathrm{Cl}(\mathfrak{f})}
\chi(C)\ln|g_\mathfrak{f}(C)|,
\end{equation}
and the \textit{$L$-function} $L_\mathfrak{f}(s,\chi)$ on $\mathbb{C}$ by
\begin{equation*}
L_\mathfrak{f}(s,\chi)=
\sum_{\mathfrak{a}}\frac{\chi([\mathfrak{a}])}{\mathrm{N}_{K/\mathbb{Q}}(\mathfrak{a})^s},
\end{equation*}
where $\mathfrak{a}$ runs over all nontrivial ideals of $\mathcal{O}_K$ prime to $\mathfrak{f}$
and $[\mathfrak{a}]$ means the class of $\mathrm{Cl}(\mathfrak{f})$ containing $\mathfrak{a}$.

\begin{proposition}\label{Kronecker}
Let $\mathfrak{f}_\chi$ be the conductor of $\chi$ and $\chi_0$ be the primitive character
of $\mathrm{Cl}(\mathfrak{f}_\chi)$ corresponding to $\chi$.
If $\mathfrak{f}_\chi\neq\mathcal{O}_K$, then we achieve
\begin{equation*}
L_{\mathfrak{f}_\chi}(1,\chi_0)
\prod_{\begin{smallmatrix}\mathfrak{p}~:~\textrm{prime ideals of $\mathcal{O}_K$}\\
~~~\textrm{such that}~\mathfrak{p}\,|\,\mathfrak{f},~\mathfrak{p}\nmid~\mathfrak{f}_\chi\end{smallmatrix}}
(1-\overline{\chi}_0([\mathfrak{p}]))=-
\frac{\pi\chi_0([\gamma\mathfrak{d}_K\mathfrak{f}_\chi])}{3N(\mathfrak{f}_\chi)\sqrt{|d_K|}\omega(\mathfrak{f}_\chi)
T_\gamma(\overline{\chi}_0)}S_{\mathfrak{f}}(\overline{\chi}),
\end{equation*}
where $\mathfrak{d}_K$ is
the different of the extension $K/\mathbb{Q}$, $\gamma$ is an element of $K$ such that
$\gamma\mathfrak{d}_K\mathfrak{f}_\chi$ is a nontrivial ideal of $\mathcal{O}_K$ prime to $\mathfrak{f}_\chi$,
$N(\mathfrak{f}_\chi)$ is the smallest positive integer in $\mathfrak{f}_\chi$,
$\omega(\mathfrak{f}_\chi)
=|\{\alpha\in\mathcal{O}_K^\times~|~\alpha\equiv1\Mod{\mathfrak{f}_\chi}\}|$ and
\begin{eqnarray*}
T_\gamma(\overline{\chi}_0)=
\sum_{x+\mathfrak{f}_\chi\in(\mathcal{O}_K/\mathfrak{f}_\chi)^\times}
\overline{\chi}_0([x\mathcal{O}_K])e^{2\pi
i\mathrm{Tr}_{K/\mathbb{Q}}(x\gamma)}.
\end{eqnarray*}
\end{proposition}
\begin{proof}
See \cite[Chapter II, Theorem 9]{Siegel} or \cite[Chapter 11, Theorem
2.1]{K-L}.
\end{proof}

\begin{remark}\label{Eulerfactor}
\begin{itemize}
\item[(i)] We have $|T_\gamma(\overline{\chi}_0)|=\sqrt{\mathrm{N}_{K/\mathbb{Q}}(\mathfrak{f}_\chi)}\neq0$
(\cite[p.107]{Siegel}).
\item[(ii)] Since $\chi_0$ is a nonprincipal character of $\mathrm{Cl}(\mathfrak{f}_\chi)$, we get $L_{\mathfrak{f}_\chi}(1,\chi_0)\neq0$
(\cite[Chapter V, Theorem 10.2]{Janusz}).
\item[(iii)] If every prime ideal factor of $\mathfrak{f}$ divides $\mathfrak{f}_\chi$,
then we understand the Euler factor $\prod_{\mathfrak{p}\,|\,\mathfrak{f},~\mathfrak{p}\nmid~\mathfrak{f}_\chi}
(1-\overline{\chi}_0([\mathfrak{p}]))$ to be $1$.
\end{itemize}
\end{remark}

\section {Generation of ring and ray class fields}\label{sectray}

Let $K$ be an imaginary quadratic field,
$\mathfrak{f}=N\mathcal{O}_K$ for an integer $N>1$
and $\mathcal{O}$ be the order of conductor $N$ in $K$.
Let $E$ be the elliptic curve stated in \textup{(\ref{parametrization})} and $h_E$
be the Weber function on $E$ defined as in \textup{(\ref{Weber})}.
By Proposition \ref{CM} (ii) and Remark \ref{WeberFricke} we know that
$h_E(\varphi_E(1/N))$ generates $K_\mathfrak{f}$ over $H_K$.
However, in this section we shall show under some condition that
one can pull the Hilbert class field $H_K$ down to the ground field $K$.
\par
For an intermediate field
$L$ of the extension $K_\mathfrak{f}/K$ we denote by $\mathrm{Cl}(K_\mathfrak{f}/L)$
the subgroup of $\mathrm{Cl}(\mathfrak{f})$ corresponding to
$\mathrm{Gal}(K_\mathfrak{f}/L)$
via the Artin reciprocity map $\sigma_\mathfrak{f}:\mathrm{Cl}(\mathfrak{f})
\stackrel{\sim}{\rightarrow}\mathrm{Gal}(K_\mathfrak{f}/K)$.

\begin{lemma}\label{character}
Assume that $K$ is different from $\mathbb{Q}(\sqrt{-1}),\mathbb{Q}(\sqrt{-3})$ and $N$ is prime to $6$. Then
there is a character $\chi$ of $\mathrm{Cl}(\mathfrak{f})$ satisfying
the following properties:
\begin{itemize}
\item[\textup{(C1)}] $\chi$ is trivial on $\mathrm{Cl}(K_\mathfrak{f}/K_\mathcal{O})$.
\item[\textup{(C2)}] $\chi(C')\neq1$ for any chosen
$C'\in\mathrm{Cl}(\mathfrak{f})\setminus\mathrm{Cl}(K_\mathfrak{f}/K_\mathcal{O})$
\textup{(}if any\textup{)}.
\item[\textup{(C3)}] Every prime ideal factor of $\mathfrak{f}$
divides the conductor $\mathfrak{f}_\chi$ of $\chi$.
\end{itemize}
\end{lemma}
\begin{proof}
See \cite[Lemma 3.4 and Remark 4.5]{K-S-Y}.
\end{proof}

\begin{theorem}\label{relativenorm}
With the same assumptions as in \textup{Lemma \ref{character}}, let
\begin{equation*}
\xi_N=h_E(\varphi_E(2/N))-h_E(\varphi_E(1/N))
\end{equation*}
which lies in $K_\mathfrak{f}$ by \textup{(\ref{Hassegenerators})}. Then
its relative norm $\mathrm{N}_{K_\mathfrak{f}/K_\mathcal{O}}(\xi_N)$ generates
$K_\mathcal{O}$ over $K$.
\end{theorem}
\begin{proof}
Let
\begin{equation*}
F=K(\mathrm{N}_{K_\mathfrak{f}/K_\mathcal{O}}(\xi_N))
\end{equation*}
which is an abelian extension of $K$ as a subfield of $K_\mathcal{O}$.
Suppose on the contrary that $F$ is properly contained in $K_\mathcal{O}$.
Then we can take a class
$C'$ in $\mathrm{Cl}(K_\mathfrak{f}/F)\setminus\mathrm{Cl}(K_\mathfrak{f}/K_\mathcal{O})$.
Let $\chi$ be a character of $\mathrm{Cl}(\mathfrak{f})$ satisfying the conditions (C1), (C2), (C3) in Lemma \ref{character}. By (C1) and (C2) we note that $\chi$ viewed as a character of $\mathrm{Cl}(K_\mathfrak{f}/F)/\mathrm{Cl}(K_\mathfrak{f}/K_\mathcal{O})$ is nontrivial.
Furthermore, it follows from Proposition \ref{Kronecker}, (C3) and Remark \ref{Eulerfactor} that
\begin{equation}\label{nonzero}
S_\mathfrak{f}(\overline{\chi})\neq0.
\end{equation}
\par
On the other hand, since $N$ is prime to $6$, we may regard
$C_2=[2\mathcal{O}_K]$ and $C_3=[3\mathcal{O}_K]$ as elements of $\mathrm{Cl}(K_\mathfrak{f}/K_\mathcal{O})$
by (\ref{Galois}). Let $C_0$ be the identity class of $\mathrm{Cl}(\mathfrak{f})$.
Then we deduce that
\begin{eqnarray}
\xi_N^{12N}&=&
(f_{\left[\begin{smallmatrix}0\\2/N\end{smallmatrix}\right]}(\tau_K)
-f_{\left[\begin{smallmatrix}0\\1/N\end{smallmatrix}\right]}(\tau_K))^{12N}
\quad\textrm{by Remark \ref{WeberFricke} (ii)}\nonumber\\
&=&\left(\frac{j(\tau_K)^2(j(\tau_K)-1728)^3}{2^{30}3^{24}}\right)^{2N}
\frac{g_{\left[\begin{smallmatrix}0\\3/N\end{smallmatrix}\right]}(\tau_K)^{12N}}
{g_{\left[\begin{smallmatrix}0\\2/N\end{smallmatrix}\right]}(\tau_K)^{24N}
g_{\left[\begin{smallmatrix}0\\1/N\end{smallmatrix}\right]}(\tau_K)^{12N}}\quad\textrm{by Lemma \ref{FrickeSiegel}}\nonumber\\
&=&\left(\frac{j(\tau_K)^2(j(\tau_K)-1728)^3}{2^{30}3^{24}}\right)^{2N}
\frac{g_\mathfrak{f}(C_3)}{g_\mathfrak{f}(C_2)^2g_\mathfrak{f}(C_0)}\quad\textrm{by definition (\ref{Frickeinvariant})}.
\label{difference}
\end{eqnarray}
Now that the assumption $K\neq\mathbb{Q}(\sqrt{-1}),\mathbb{Q}(\sqrt{-3})$ implies that
$j(\tau_K)\neq0,1728$ by Remark \ref{WeberFricke} (i), this value is nonzero.
\par
If $\chi$ is trivial on $\mathrm{Cl}(K_\mathfrak{f}/H_K)$, then
$\chi$ can be viewed as a character of $\mathrm{Cl}(\mathfrak{f})/\mathrm{Cl}(K_\mathfrak{f}/H_K)\simeq
\mathrm{Gal}(H_K/K)\simeq\mathrm{Cl}(\mathcal{O}_K)$.
This yields $\mathfrak{f}_\chi=\mathcal{O}_K$, which contradicts (C3). Thus there exists a class $C''\in\mathrm{Cl}(K_\mathfrak{f}/H_K)$
such that $\chi(C'')\neq1$.
Now, consider the element $x$ of $F$ given by
\begin{equation*}
x=\frac{\mathrm{N}_{K_\mathfrak{f}/K_\mathcal{O}}(\xi_N)^{12N}}
{(\mathrm{N}_{K_\mathfrak{f}/K_\mathcal{O}}(\xi_N)^{12N})^{\sigma_\mathfrak{f}(C'')}}
=\mathrm{N}_{K_\mathfrak{f}/K_\mathcal{O}}\left(\frac{\xi_N^{12N}}
{(\xi_N^{12N})^{\sigma_\mathfrak{f}(C'')}}\right).
\end{equation*}
Since $\sigma_\mathfrak{f}(C'')$ leaves $H_K=K(j(\tau_K))$ fixed,
we attain by (\ref{difference}) and Proposition \ref{invariant} that
\begin{eqnarray*}
x&=&\mathrm{N}_{K_\mathfrak{f}/K_\mathcal{O}}\left(\frac{g_\mathfrak{f}(C_3)}{g_\mathfrak{f}(C_2)^2g_\mathfrak{f}(C_0)}
\left(\frac{g_\mathfrak{f}(C_2)^2g_\mathfrak{f}(C_0)}{g_\mathfrak{f}(C_3)}\right)^{\sigma_\mathfrak{f}(C'')}\right)\\
&=&\mathrm{N}_{K_\mathfrak{f}/K_\mathcal{O}}\left(\frac{g_\mathfrak{f}(C_3)}{g_\mathfrak{f}(C_2)^2g_\mathfrak{f}(C_0)}
\frac{g_\mathfrak{f}(C_2C'')^2g_\mathfrak{f}(C'')}{g_\mathfrak{f}(C_3C'')}\right).
\end{eqnarray*}
Set
\begin{equation*}
y=\frac{g_\mathfrak{f}(C_3)}{g_\mathfrak{f}(C_2)^2g_\mathfrak{f}(C_0)}
\frac{g_\mathfrak{f}(C_2C'')^2g_\mathfrak{f}(C'')}{g_\mathfrak{f}(C_3C'')}
\end{equation*}
so that $x=\mathrm{N}_{K_\mathfrak{f}/K_\mathcal{O}}(y)$.
We then derive that
\begin{eqnarray}
&&\sum_{C\in\mathrm{Cl}(\mathfrak{f})}\overline{\chi}(C)\ln
\left|y^{\sigma_\mathfrak{f}(C)}\right|\nonumber\\
&=&\sum_{C\in\mathrm{Cl}(\mathfrak{f})}\overline{\chi}(C)\ln
\left|\frac{g_\mathfrak{f}(C_3C)}{g_\mathfrak{f}(C_2C)^2g_\mathfrak{f}(C)}
\frac{g_\mathfrak{f}(C_2C''C)^2g_\mathfrak{f}(C''C)}{g_\mathfrak{f}(C_3C''C)}\right|
\quad\textrm{by Proposition \ref{invariant}}\nonumber\\
&=&(\chi(C_3)-2\chi(C_2)-1)(1-\chi(C''))\sum_{C\in\mathrm{Cl}(\mathfrak{f})}\overline{\chi}(C)\ln|g_\mathfrak{f}(C)|\nonumber\\
&=&-2(1-\chi(C''))S_\mathfrak{f}(\overline{\chi})\nonumber\\
&&\textrm{by definition (\ref{Stickelberger}),
(C2) and the fact
$C_2,C_3\in\mathrm{Cl}(K_\mathfrak{f}/K_\mathcal{O})$}.\label{derive}
\end{eqnarray}
On the other hand, we know that $(1-\chi(C''))$ is nonzero.
And, we also get that
\begin{eqnarray*}
&&\sum_{C\in\mathrm{Cl}(\mathfrak{f})}\overline{\chi}(C)\ln|y^{\sigma_\mathfrak{f}(C)}|\\
&=&
\sum_{\begin{smallmatrix}E\in\mathrm{Cl}(\mathfrak{f})\\
E\Mod{\mathrm{Cl}(K_\mathfrak{f}/F)}
\end{smallmatrix}}
\sum_{\begin{smallmatrix}D\in\mathrm{Cl}(K_\mathfrak{f}/F)\\
D\Mod{\mathrm{Cl}(K_\mathfrak{f}/K_\mathcal{O})}
\end{smallmatrix}}
\sum_{C\in\mathrm{Cl}(K_\mathfrak{f}/K_\mathcal{O})}
\overline{\chi}(CDE)
\ln|y^{\sigma_\mathfrak{f}(CDE)}|\\
&=&\sum_{\begin{smallmatrix}E\in\mathrm{Cl}(\mathfrak{f})\\
E\Mod{\mathrm{Cl}(K_\mathfrak{f}/F)}
\end{smallmatrix}}\overline{\chi}(E)
\sum_{\begin{smallmatrix}D\in\mathrm{Cl}(K_\mathfrak{f}/F)\\
D\Mod{\mathrm{Cl}(K_\mathfrak{f}/K_\mathcal{O})}
\end{smallmatrix}}\overline{\chi}(D)
\sum_{C\in\mathrm{Cl}(K_\mathfrak{f}/K_\mathcal{O})}
\ln|y^{\sigma_\mathfrak{f}(C)\sigma_\mathfrak{f}(D)\sigma_\mathfrak{f}(E)}|
\quad\textrm{by (C1)}\\
&=&\sum_{\begin{smallmatrix}E\in\mathrm{Cl}(\mathfrak{f})\\
E\Mod{\mathrm{Cl}(K_\mathfrak{f}/F)}
\end{smallmatrix}}\overline{\chi}(E)
\sum_{\begin{smallmatrix}D\in\mathrm{Cl}(K_\mathfrak{f}/F)\\
D\Mod{\mathrm{Cl}(K_\mathfrak{f}/K_\mathcal{O})}
\end{smallmatrix}}\overline{\chi}(D)
\ln|x^{\sigma_\mathfrak{f}(D)\sigma_\mathfrak{f}(E)}|
\quad\textrm{since}~x=\mathrm{N}_{K_\mathfrak{f}/K_\mathcal{O}}(y)\\
&=&\sum_{\begin{smallmatrix}E\in\mathrm{Cl}(\mathfrak{f})\\
E\Mod{\mathrm{Cl}(K_\mathfrak{f}/F)}
\end{smallmatrix}}\overline{\chi}(E)
\ln|x^{\sigma_\mathfrak{f}(E)}|
\sum_{\begin{smallmatrix}D\in\mathrm{Cl}(K_\mathfrak{f}/F)\\
D\Mod{\mathrm{Cl}(K_\mathfrak{f}/K_\mathcal{O})}
\end{smallmatrix}}\overline{\chi}(D)
\quad\textrm{by the fact}~x\in F\\
&=&0\quad\textrm{because $\chi$
as a character of $\mathrm{Cl}(K_\mathfrak{f}/F)/\mathrm{Cl}(K_\mathfrak{f}/K_\mathcal{O})$
is nontrivial},
\end{eqnarray*}
from which we obtain $S_\mathfrak{f}(\overline{\chi})=0$ by (\ref{derive}).
But, this contradicts (\ref{nonzero}).
\par
Therefore, we conclude $F=K_\mathcal{O}$ as desired.
\end{proof}

\begin{remark}\label{explicit}
We see that
\begin{eqnarray}
\xi_N&=&h_E(\varphi_E(2/N))-h_E(\varphi_E(1/N))\nonumber\\
&=&f_{\left[\begin{smallmatrix}0\\2/N\end{smallmatrix}\right]}(\tau_K)-
f_{\left[\begin{smallmatrix}0\\1/N\end{smallmatrix}\right]}(\tau_K)\quad\textrm{by Remark \ref{WeberFricke} (ii)}\nonumber\\
&=&f_\mathfrak{f}([2\mathcal{O}_K])-f_\mathfrak{f}([\mathcal{O}_K])\quad\textrm{by definition (\ref{Frickeinvariant})}.\label{xiN}
\end{eqnarray}
On the other hand, we have the isomorphism
\begin{eqnarray*}
(\mathbb{Z}/N\mathbb{Z})^\times/\{\pm1\}&\rightarrow&\mathrm{Cl}(K_\mathfrak{f}/K_\mathcal{O})\\
t&\mapsto&[t\mathcal{O}_K]
\end{eqnarray*}
(\cite[Proposition 3.8]{E-K-S}).
Thus we can express $\mathrm{N}_{K_\mathfrak{f}/K_\mathcal{O}}(\xi_N)$ explicitly as
\begin{eqnarray*}
\mathrm{N}_{K_\mathfrak{f}/K_\mathcal{O}}(\xi_N)&=&
\prod_{t\in(\mathbb{Z}/N\mathbb{Z})^\times/\{\pm1\}}(f_\mathfrak{f}([2\mathcal{O}_K])-f_\mathfrak{f}([\mathcal{O}_K]))
^{\sigma_\mathfrak{f}([t\mathcal{O}_K])}
\quad\textrm{by (\ref{xiN})}\\
&=&\prod_{t\in(\mathbb{Z}/N\mathbb{Z})^\times/\{\pm1\}}
(f_\mathfrak{f}([2t\mathcal{O}_K])-f_\mathfrak{f}([t\mathcal{O}_K]))\quad\textrm{by Proposition \ref{invariant}}\\
&=&\prod_{t\in(\mathbb{Z}/N\mathbb{Z})^\times/\{\pm1\}}
(f_{\left[\begin{smallmatrix}0\\2t/N\end{smallmatrix}\right]}(\tau_K)-
f_{\left[\begin{smallmatrix}0\\t/N\end{smallmatrix}\right]}(\tau_K))\quad\textrm{by definition (\ref{Frickeinvariant})}\\
&=&\prod_{t\in(\mathbb{Z}/N\mathbb{Z})^\times/\{\pm1\}}
(h_E(\varphi_E(2t/N))-h_E(\varphi_E(t/N)))\quad\textrm{by Remark \ref{WeberFricke} (ii)}.
\end{eqnarray*}

\end{remark}

\begin{corollary}\label{main}
If $N$ is prime to $6$, then
$h_E(\varphi_E(1/N))$ generates $K_\mathfrak{f}$ over $K$.
\end{corollary}
\begin{proof}
If $K=\mathbb{Q}(\sqrt{-1})$ or $\mathbb{Q}(\sqrt{-3})$, then
the assertion holds by Proposition \ref{CM} (ii) and Remark \ref{WeberFricke} (i).
\par
Thus we let $K\neq\mathbb{Q}(\sqrt{-1}),\mathbb{Q}(\sqrt{-3})$, and
set $L=K(h_E(\varphi_E(1/N)))$. Since $L$
is an abelian extension of $K$ as a subfield of $K_\mathfrak{f}$ by
(\ref{Hassegenerators}), it
also contains $h_E(\varphi_E(1/N))^{\sigma_\mathfrak{f}([2\mathcal{O}_K])}$.
Here, we note that since $N$ is prime to $6$, $\mathrm{Cl}(\mathfrak{f})$ contains the class $[2\mathcal{O}_K]$.
We find that
\begin{eqnarray*}
h_E(\varphi_E(1/N))^{\sigma_\mathfrak{f}([2\mathcal{O}_K])}&=&
f_{\left[\begin{smallmatrix}0\\1/N\end{smallmatrix}\right]}(\tau_K)^{\sigma_\mathfrak{f}([2\mathcal{O}_K])}
\quad\textrm{by Remark \ref{WeberFricke} (ii)}\\
&=&f_\mathfrak{f}([\mathcal{O}_K])^{\sigma_\mathfrak{f}([2\mathcal{O}_K])}
\quad\textrm{by definition (\ref{Frickeinvariant})}\\
&=&f_\mathfrak{f}([2\mathcal{O}_K])\quad\textrm{by Proposition \ref{invariant}}\\
&=&f_{\left[\begin{smallmatrix}0\\2/N\end{smallmatrix}\right]}(\tau_K)\quad
\textrm{by definition (\ref{Frickeinvariant})}\\
&=&h_E(\varphi_E(2/N))
\quad\textrm{by Remark \ref{WeberFricke} (ii)}.
\end{eqnarray*}
Thus we deduce that
\begin{eqnarray*}
L&=&K(\xi_N,h_E(\varphi_E(1/N)))\quad\textrm{where}~
\xi_N=h_E(\varphi_E(2/N))-h_E(\varphi_E(1/N))\\
&=&K(\mathrm{N}_{K_\mathfrak{f}/K_\mathcal{O}}(\xi_N),h_E(\varphi_E(1/N)))\\
&=&K_\mathcal{O}(h_E(\varphi_E(1/N)))\quad\textrm{by Theorem \ref{relativenorm}}\\
&=&K_\mathfrak{f}\quad\textrm{by (\ref{classinclusion}) and Proposition \ref{CM} (ii)}.
\end{eqnarray*}
Therefore, we achieve that $L=K_\mathfrak{f}$.
\end{proof}

\bibliographystyle{amsplain}

\address{% the first author
Department of Mathematical Sciences \\
KAIST \\
Daejeon 305-701 \\
Republic of Korea} {jkkoo@math.kaist.ac.kr}
\address{% the corresponding author
Department of Mathematics\\
Hankuk University of Foreign Studies\\
Yongin-si, Gyeonggi-do 449-791\\
Republic of Korea} {dhshin@hufs.ac.kr}
\address{
National Institute for Mathematical
Sciences\\
Daejeon 305-811\\
Republic of Korea} {dsyoon@nims.re.kr}

\end{document}